\documentclass{article}
\usepackage[margin=1.5 in]{geometry}
\usepackage{comment}
\usepackage{graphics}	
\usepackage[utf8]{inputenc}
\usepackage[english]{babel}										%
\usepackage{amsmath}
\usepackage{tikz}
\usetikzlibrary{cd}
\usepackage{enumerate}
\usepackage{bbm}
\usepackage{tikz-cd}
\tikzcdset{scale cd/.style={every label/.append style={scale=#1},
    cells={nodes={scale=#1}}}}
\usepackage{amssymb}
\usepackage{mathrsfs}
\tikzcdset{scale cd/.style={every label/.append style={scale=#1},
    cells={nodes={scale=#1}}}}
\usepackage{amsthm}
\newtheorem{theorem}{Theorem}
\newtheorem{lemma}[theorem]{Lemma}
\newtheorem{prop}[theorem]{Proposition}
\newtheorem{ex}[theorem]{Example}
\newtheorem{deff}[theorem]{Definition}

\newtheorem{rmk}[theorem]{Remark}

\newcommand{\catC}{\mathscr{C}}
\newcommand{\Res}{\text{Res}}

\newcommand{\Hom}{\text{Hom}}

\newcommand{\Z}{\mathbb{Z}}

\newcommand{\Sym}{\text{Sym}}

\newcommand{\id}{\mathrm{id}}

\newcommand{\Ob}{\mathsf{Ob}}
\newcommand{\Func}{\mathsf{Func}}

\newcommand{\com}{\text{com}}
\newcommand{\catD}{\mathscr{D}}

\newcommand{\Ind}{\text{Ind}}
\newcommand{\catB}{\mathscr{B}}

\begin{document}

\author{Ville Nordström}
\title{A decomposition theorem for the Hochschild homology of symmetric powers of a dg category}
\date{}
\maketitle
\begin{abstract} We prove a conjecture by Belmans, Fu and Krug concerning the Hochschild homology of the symmetric powers of a small dg category $\catC$. More precisely, we show that these groups decompose into pieces that only depend on the Hochschild homology of the dg category $\catC$. 
\end{abstract}
\section{Introduction}
In \cite{belman} it was shown (see corollary D) that for any smooth projective variety $X$ over a field of characteristic zero we have
$$\bigoplus_{n\geq 0} HH_\bullet([\Sym^n(X)])t^n\cong \Sym^\bullet(\bigoplus_{i\geq 1}HH_\bullet(X)t^i)$$
where $[\Sym^n(X)]$ refers to the quotient stack $[X^n/S_n]$. The notation on the right refers to the graded symmetric algerba, where the grading on $\bigoplus_{i\geq 1}HH_\bullet(X)t^i$ is the usual grading on Hochschild homology and the formal variable $t$ is of degree $0$.
The authors then formulate the more general conjecture (see conjecture 3.24 in \cite{belman}) that for any smooth proper dg category $\catC$ we have an isomorphism
\begin{align}\label{iso}\bigoplus_{n\geq 0} HH_\bullet(\Sym^n(\catC))t^n\cong \Sym^\bullet(\bigoplus_{i\geq 1}HH_\bullet(\catC)t^i)\end{align}
where $\Sym^n(\catC)$ denotes the equivariant dg category associated to the natural $S_n$ action on the pretriangulated hull of $\catC^{\otimes n}$.
In this paper we show that in fact, there is an isomorphism as above for any small dg category $\catC$, and moreover this isomorphism is natural in $\catC$. The proof of this has three steps. First, we have the following decomposition 
$$HH_\bullet(\Sym^n(\catC))\cong \oplus_{[\sigma]\in S_n/\sim} HH_\bullet(\catC^{\otimes n},\sigma)^{C(\sigma)}$$
where $S_n/\sim$ means conjugacy classes (see \cite{ville3}). After identifying conjugacy classes in $S_n$ with partitions $n=\lambda_1+...+\lambda_k$ we can choose the permutations to be of the form $$\sigma=(1\quad 2\quad \cdots \quad \lambda_1)(\lambda_1+1\quad \cdots \quad\lambda_1+\lambda_2)\cdots (\lambda_1+...+\lambda_{k-1}+1\quad\cdots \quad n).$$
The second step is then a version of the Künneth isomorphism
$$HH_\bullet(\catC^{\otimes n},\sigma)\cong HH_\bullet(\catC^{\otimes \lambda_1},\sigma_{\lambda_1})\otimes...\otimes HH_\bullet(\catC^{\otimes \lambda_k},\sigma_{\lambda_k})$$
where $\sigma_{\lambda_i}$ is a generator for the cyclic group of order $\lambda_i$. The final step is the computation that $HH_\bullet(\catC^{\otimes i},\sigma_i)\cong HH_\bullet(\catC)$ for any $i$ (see proposition \ref{keycomputation}) which implies that the above can be written as
$$HH_\bullet(\catC^{\otimes n},\sigma)\cong HH^\bullet(\catC^{\otimes \lambda_1},\sigma_{\lambda_1})\otimes...\otimes HH_\bullet(\catC^{\otimes \lambda_k},\sigma_{\lambda_k})\cong HH_\bullet(\catC)^{\otimes k}.$$
After taking invariants with respect to the $C(\sigma)-$action we obtain a summand of \\$\Sym^\bullet(\oplus_{i\geq 1}HH_\bullet(\catC)t^i)$. The proof then reduces to checking that when we vary $n$ and $\sigma\in S_n$ we get all of $\Sym^\bullet(\oplus_{i\geq 1}HH_\bullet(\catC)t^i)$.\\

After writing this paper we found out that Anno, Baranovsky and Logvinenko have also proved this conjecture. They provide explicit maps between standard complexes and show that they are homotopy equivalences. The formulas for the maps and the summary of the results appeared in \cite{Logvinenko-Gyenge} and the proofs are in \cite{AnnoBaranovskyLogvinenko}, along with the description of the Hopf algebra structure induced on the decomposition.

Gyenge, Koppensteiner and Logvinenko have previously also studied a similar decomposition from a $K$-theoretic perspective, see \cite{gyen-kopp-log} theorem C.

\subsection{Plan of the paper}
In section \ref{warmup} we carry out a warm up computation that we hope will facilitate reading the proof of the key computation proposition \ref{keycomputation}.

In sections \ref{dg categories and functors}-\ref{pretr} we recall the definitions and basic facts about dg categories and their modules.

In section \ref{Finite group actions on dg categories} we recall the notion of group actions on dg categories that is relevant to us.

In section \ref{hh} we recall the definition of and some basic facts about Hochschild homology.

In section \ref{hhforeq} we summarize the main result from \cite{ville3}.

In section \ref{main theorem} we prove our main result about the existence of a decomposition of the form \ref{iso} for any small dg category.
\subsection{Acknowledgments}
The author is grateful to Alexander Polishchuk, Greg Stevenson and Clas Löfwall for helpful discussions. The author was supported by the DFF (Independent Research Fund Denmark) grant no. 10.46540/4283-00116B.

\subsection{Conventions}
We work over a fixed field $k$ of characteristic zero.

Our conventions on matrices are that they act on the left. So whenever the source and target of some morphism are decomposed as direct sums, that morphism can be described as a matrix whose columns are indexed by the summands of the source and the rows are indexed by the summands of the target.

If $V$ is a graded vector space the free graded commutative algebra on $V$ can be described either as a subalgebra of the tensor algebra $T^\bullet(V)$ or as a quotient of it. To distinguish between these we will use $\Sym^\bullet(V)$ to denote the subalgebra of $T^\bullet(V)$ and $S^\bullet(V)$ for the quotient algebra.

\subsection{Warm up computation}\label{warmup}
The notation in the computation $HH_\bullet(\catC^{\otimes i},\sigma_i)\cong HH_\bullet(\catC)$ (which is proposition \ref{keycomputation}) gets a little complicated when working with categories and it becomes hard to see the actual reason why it works. Therefore, as a warm up, we prove this fact here in the special case that $\catC=A$ is an associative algebra and $i=2$.

So, let $A$ be an associative algebra. We write $A^e:=A\otimes A^{op}$ and will use lower case letters $f,g,..$ for elements in $A$ and primed letters $f',g'$ for elements in $A^{op}$. Let $A_2=A^{\otimes 2}$ and consider the ring automorphism $\sigma:A_2\to A_2$, $(a,b)\mapsto (b,a)$. Write $A_2^e=A_2\otimes A_2^{op}$. Then $A_2$ is naturally a left $A_2^e$ module via $(f_1,f_2,f_1',f_2')\cdot (a,b):=(f_1af_1',f_2bf_2')$. We have the following free resolution of the left $A_2^e$-module $A_2$ 
$$Bar(A_2):=[...\to A_2^{\otimes 3}\to A_2^{\otimes 2}]\to A_2.$$
Let $^\sigma A_2$ be the right $A_2^e$-module whose underlying vector space is $A_2$ and whose module structure is given by
$(a,b)\cdot (f_1,f_2,f_1',f_2'):=(f_2'af_1,f_1'bf_2)$. We have four ring maps $A^e\to A_2^e$ defined by
$$e_{11}:(f,g')\mapsto (f,1,g',1'), e_{12}:(f,g')\mapsto (f,1,1',g'), $$
$$e_{21}:(f,g')\mapsto (1,f,g',1'), e_{22}:(f,g')\mapsto (1,f,1',g').$$
If $M$ is a left $A_2^e$ module and $N$ is a right $A^e$ module we write $N\otimes_{A^e_{21}}M:=N\otimes_{A^e}e_{21}^*(M)$. 
In other words $N\otimes_{A^e_{21}}M$ is the cokernel of the map $N\otimes (A\otimes A^{op})\otimes M\to N\otimes M$, 
$$(n,f\otimes g',m)\mapsto (n(f\otimes g'),m)-(n,(1\otimes f\otimes g'\otimes 1')m).$$
Note that since elements in the image of $e_{21}$ commute with elements in the image of $e_{12}$ this tensor product inherits a $A^e$-module structure via $e_{12}$.
\begin{lemma}
    For any left $A_2^e$ module $M$ we have an isomorphism 
    $$^\sigma A_2\otimes_{A_2^e}M\cong A\otimes_{A^e_{12}} (A\otimes_{A^e_{21}}M).$$
\end{lemma}
\begin{proof}
    Consider the diagram
    $$\begin{tikzcd}
        (A\otimes A)\otimes M\arrow[two heads]{rr}{\pi}\arrow{d}{=}&&^\sigma A_2\otimes_{A_2^e}M \\
        A\otimes(A\otimes M)\arrow[two heads]{d}{p_1}&&\\
        A\otimes (A\otimes_{A_{21}^e} M)\arrow[two heads]{rr}{p_2}&&A\otimes_{A_{12}^e}(A\otimes_{A_{21}^e}M).
    \end{tikzcd}$$
We will show that the kernels of $\pi$ and $p_2p_1$ are the same. Suppose 
    $x\in\ker(\pi)$. Then $x$ can be written as a sum of elements of the form
    $$\delta((a,b)\otimes (f_1,f_2,f_1',f_2')\otimes m):=(f_2'af_1,f_1'bf_2)\otimes m-(a,b)\otimes (f_1\otimes f_2)m(f_1'\otimes f_2').$$
Such elements vanish under $p_2p_1$. For the other inclusion suppose $x\in \ker(p_2p_1)$. Let $\bar{x}=p_1(x)$ which is then in the kernel of $p_2$. Therefore $\bar{x}$ can be written as a sum of elements of the form
$$f'af\otimes (b\otimes_{A_{21}^e} m)-a\otimes (b\otimes_{A^e_{21}}(f\otimes 1)m(1\otimes f')).$$
Let $x'$ denote a lift of $\bar{x}$ which is the sum of elements of the form
$$f'af\otimes (b\otimes m)-a\otimes (b\otimes (f\otimes 1)m(1\otimes f')).$$
Then $\pi$ vanishes on $x'$. Also $p_1$ vanishes on $x-x'$ so $x-x'$ can be written as as sum of elements of the form
$$a\otimes (f'bf\otimes m-b\otimes (1\otimes f)m(f'\otimes 1))$$
and $\pi$ vanishes on such elements. Therefore $\pi(x)=0$.
\end{proof}
Now, $HH_\bullet(A_2,\sigma)$ is obtained by applying 
$^\sigma A_2\otimes_{A_2^e} (-)$ to the complex $Bar(A_2)$. By the lemma this is isomorphic to what we get by first applying $A\otimes_{A^e_{21}}(-)$ and then applying $A\otimes_{A_{12}^e}(-)$. But here is the key observation: $e_{21}^*(A_2^{\otimes n})$ is free as a $A^e$ module for all $n$ including $n=1$. (For $n=1$ the same is not true for $e_{11}$ or $e_{22}$!) Moreover for $n\geq 2$ $A\otimes_{A^e_{21}}A_2^{\otimes n}$ are free $A^e$-modules via $e_{12}$. Therefore $A\otimes_{A^e_{21}}Bar(A_2)$ is a free resolution of $A$ as a left $A^e$-module. After applying $A\otimes_{A^e_{12}}(-)$ we get $HH_\bullet(A)$. We have proved the following:
\begin{prop}
    Let $A$ be any associative algebra. Then $HH_\bullet(A_2,\sigma)\cong HH_\bullet(A)$.
\end{prop}

\section{Dg categories}
We recall some basic facts about dg categories. Nothing here is new but some proofs are included for convenience. 
\subsection{Dg categories and dg functors} \label{dg categories and functors}
A \textit{dg category} $\catC$ is a category enriched in $\Z$-graded chain complexes. Here are some basic examples which we will encounter a lot. More interesting examples will be treated later.

\begin{ex}\label{com} We denote by $\com(k)$ the category whose objects are $\Z$-graded chain complexes over $k$ and whose homomorphism complexes are defined as 
$$\Hom_{\com(k)}(V,W)^i=\prod_{j\in \Gamma}\Hom_k(V^j,W^{j+i})$$
and the differentials are given by 
$$\partial(f)=d_W\circ f-(-1)^{|f|}f\circ d_V.$$
\end{ex}
\begin{ex}
A $\Z$-graded dg algebra can be thought of as a dg category with one object.
\end{ex}

Given a dg category $\catC$ we can associate to it four new categories $Z^0(\catC)$, $H^0(\catC)$, $H^\bullet(\catC)$ and $\catC^\#$ where the last two are graded category. The objects of each of these are the same as the objects of $\catC$. The homomorphism spaces are obtained from those in $\catC$ by taking degree zero cycles, taking zero'th homology, taking the total homology and forgetting the differentials respectively.

Let $\catC$ and $\catC'$ be dg categories. A \textit{dg functor} $F:\catC\to \catC'$ is a functor between categories enriched in chain complexes. A dg functor $F:\catC\to \catC'$ induces functors $Z^0(F): Z^0(\catC)\to Z^0(\catC'),$ $H^0(F):H^0(\catC)\to H^0(\catC')$, $H^\bullet(F):H^\bullet(\catC)\to H^\bullet(\catC')$, $F^\#:\catC^\#\to \catC'^\#$. We say that $F$ is a \textit{quasi-equivalence} if $H^\bullet(F)$ is fully faithful and $H^0(F)$ is essentially surjective.

The collection of dg functors from $\catC$ to $\catC'$ form the objects of a new dg category which we will denote $\Func_{dg}(\catC,\catC')$. Given two dg functors $F,G:\catC\to \catC'$, a \textit{pre-natural transformation} $\alpha:F\to G$ of degree $i$ consists of a degree $i$ morphism $\alpha_x\in \catC'(F(x),G(x))^i$ for every $x\in \Ob(\catC)$ satisfying the usual naturality conditions. The homomorphism complex $\Hom_{\Func_{dg}}(F,G)$ is the graded vector space whose degree $i$ part consists of the pre-natural transformations of degree $i$. The differential is given pointwise: $\partial(\alpha)_x=\partial(\alpha_x)$ where on the right $\partial$ refers to the differential in the complex $\catC'(F(x),G(x))$. The \textit{natural transformations} are the degree zero cycles in $\Hom_{\Func_{dg}}(F,G)$. A natural transformation is called a natural isomorphism if $\alpha_x$ is an isomorphism in $Z^0(\catC')$ for all $x\in \Ob(\catC)$.

Finally we observe that if $\catC$ is a dg category, then there is a dg category $\catC^{op}$ with $\Ob(\catC^{op})=\Ob(\catC)$ and $\catC^{op}(c,c')=\catC(c',c)$. Composition is defined by $f\circ^{op}g:=(-1)^{|f||g|}g\circ f$.

\subsection{Dg modules}\label{dgmodules}
We let $mod^{dg}-\catC$ be the dg category $\Func_{dg}(\catC^{op},\com(k))$ and we refer to its objects as \textit{right $\catC$-modules}. Left modules are defined as the objects of $\Func_{dg}(\catC,\com(k))$. To simplify notation we will denote the homomorphism complexes between two right (or left) modules simply by $\Hom_\catC(M,N)$.

We note here that the category $Z^0(mod-\catC)$ is abelian and kernels and cokernels are computed pointwise in the abelian category $Z^0(\com(k))$. By a submodule $M\subset N$ we mean a monomorphism $M\to N$ in $Z^0(mod-\catC)$.

Any module $M\in mod-\catC$ gives rise in a natural way to an underlying graded module $M^\#:\catC^\#\to \text{grVect}$.

There is a dg functor $Y:\catC\to mod^{dg}-\catC$ which we refer to as the dg-Yoneda embedding. The dg version of the usual Yoneda lemma says that there is an isomorphism of chain complexes
$$\Hom_{\Func_{dg}}(Y(x),F)\cong F(x).$$
For $x\in \catC$ let $h^x:=Y(x)$. A \textit{free module} is one of the form $\bigoplus_{i\in I}h^{x_i}\otimes V_i$ where $V_i$ are chain complexes of vector spaces. If $I$ is finite and each $V_i$ is a finite dimensional chain complex then we say that the module is \textit{finitely generated free}. We say that a  module $M$ is \textit{semi-free} if it admits a filtration $$0=F_0\subset F_1\subset F_2\subset... \subset M$$
such that for all $i\geq 1$ the quotients $F_i/F_{i-1}$ are free and $M=\cup_{i\geq 0} F_i$. 

Similarly there is a contravariant Yoneda embedding $Y':\catC^{op}\to \catC-mod$. We denote by $h_x:=Y'(x)$. Now free modules and semi-free modules are defined analogously to the case of right modules.


For any dg category $\catC$ the category $H^0(mod^{dg}-\catC)$ can be given the structure of a triangulated category. Shifts are defined pointwise: given a right module $M$ we define $M[1](x):=M(x)[1]$. Cones of morphisms are also constructed pointwise: if $\alpha:M\to M'$ is a natural transformation we set $Cone(\alpha)(x):=Cone(\alpha_x)$. The module $Cone(\alpha)$ fits naturally into a diagram in $Z^0(\catC)$
$$M'\to Cone(\alpha)\to M[1]\to M'[1].$$
The distinguished triangles in $H^0(mod^{dg}-\catC)$ are those which are isomorphic to a sequence of the above form. 

A right $\catC$-module $M$ is called acyclic if it is acyclic pointwise. A natural transformation $\alpha: M\to N$ is quasi isomorphism if it is so pointwise. Equivalently, $\alpha$ is a quasi isomorphism if and only if $Cone(\alpha)$ is acyclic.

A module $P$ is called \textit{h-projective} if $\Hom_\catC(P,N)$ is acyclic for any acyclic module $N$.
\begin{prop}\label{adapted2}
(i) If $\catC$ is a small dg category then any right (or left) $\catC$-module $M$ admits a quasi isomorphism from a semi-free module.\\

(ii) Any semi-free module is h-projective.
\end{prop}
\begin{proof}
This is theorem 3.1 in \cite{keller2}.
\end{proof}

If $\catC$ and $\catC'$ are two dg categories we define their tensor product $\catC\otimes\catC'$ to be the dg category whose objects are pairs $(c,c')$ with $c\in\Ob(\catC)$ and $c'\in \Ob(\catC')$. The homomorphism complexes are defined as
$$(\catC\otimes \catC')((c_1,c_1'),(c_2,c_2')):=\catC(c_1,c_2)\otimes \catC'(c_1',c_2').$$

The dg category of $\catC-\catC'$ \textit{bimodules} is by definition the dg category $\catC\otimes \catC'^{op}-mod$ (which is the same as $mod-\catC^{op}\otimes \catC'$) and its objects will be referred to as $\catC-\catC'$-\textit{bimodules}.

Now suppose that $\catC, \catC'$ and $\catC''$ are small dg categories. If $M$ is a $\catC-\catC'$-bimodule and $N$ is a $\catC'-\catC''$-bimodule we define their \textit{tensor product} $M\otimes_{\catC'}N$ to be the $\catC-\catC''$ bimodule which assigns to a pair $(c,c'')$ the chain complex which is the cokernel of the following map:
\begin{align*}\bigoplus_{c_1',c_2'\in \Ob(\catC')}M(c,c_2')\otimes \catC'(c_1',c_2')\otimes N(c_1',c'')\to& \bigoplus_{c'\in \Ob(\catC')}M(c,c')\otimes N(c',c'')\\
m\otimes f\otimes n\mapsto & mf\otimes n-m\otimes fn\end{align*}
Note that in the special case that $\catC=\catC''=k$ the tensor product $M\otimes_{\catC'}N$ is a $k-k$-bimodule, i.e. just a chain complex. 
We say that a $\catC-\catC'$ bimodule $M$ is \textit{free as $\catC$-module} if for each $c'\in \catC'$ we have $M(-,c')$ is free. We say that it is \textit{free as a $\catC'$} module if for all $c\in \catC$ $M(c,-)$ is free. Similarly, we say that $M$ is \textit{semi-free as a left $\catC$-module} if $M(-,c')$ is semi-free for all $c'\in\catC'$ and we say that it is \textit{semi-free as a right $\catC'$-module} if for all $c\in\catC$ $M(c,-)$ is semi-free.
\begin{lemma}\label{lemmaontensoringfree}
    Let $M\overset{\sim}{\to} N$ be a quasi isomorphism between $\catC-\catC'$-bimodules and assume that $M$ and $N$ are both semi-free as left $\catC$-modules. Then for any right $\catC$-module $L$ the morphism of right $\catC'$-modules
    $$L\otimes_\catC M\to L\otimes_\catC N$$
    is a quasi isomorphism.
\end{lemma}
\begin{proof}
    Since being a quasi isomorphism is a pointwise property we may assume that $\catC'=k$ and then $M$ and $N$ are semi-free. Then the statement follows from the fact that semi-free modules are h-flat, meaning tensoring with them preserves acyclicity or equivalently preserve quasi isomorphisms (see \cite{keller2}, lemma 6.1). Indeed, if we choose a semi-free replacement $B\overset{\sim}{\to} L$ then the following commutative diagram completes the proof
    $$\begin{tikzcd}
        B\otimes_\catC M\arrow{r}{\sim}\arrow{d}{\sim}&B\otimes_\catC N\arrow{d}{\sim}\\
        L\otimes_\catC M\arrow{r}{}&L\otimes_\catC N.
    \end{tikzcd}$$
\end{proof}

If $F:\catC\to \catD$ is a dg functor we get induced functors $F^*:mod-\catD\to mod-\catC$ and $F_*:mod-\catC\to mod-\catD$ (and similarly for left modules). The functor $F_*$ is defined by 
$$F_*(M)(d)=M\otimes_\catC F^*(\catD(d,-)).$$

\begin{prop}\label{proponoriginaltensorproduct}
Let $M\in mod-\catC$ and $N\in \catC-mod$. If $M=h^c$ is free then $M\otimes_\catC N\cong N(c)$ and if $N=h_c$ then $M\otimes_\catC N\cong M(c)$.\\







\end{prop}
\begin{proof}
Define a map 
$$\bigoplus_{x\in \catC}\catC(x,c)\otimes N(x)\to N(c)$$
by $f\otimes n\mapsto fn$ where $fn$ is short hand for $N(f)(n)$. This map descends to a map $h^c\otimes_\catC N\to N(c)$ which is an isomorphism with inverse given by $N(c)\ni n\mapsto [\id_c\otimes n]$. The case when $N=h_c$ is identical.

\end{proof}

\subsection{Pre-triangulated dg categories}\label{pretr}
Given a morphism $\alpha:x\to y$ in $Z^0(\catC)$ the cone of $\alpha$ refers to any object $z\in \catC$ such that $h^z\cong C(Y(\alpha))$ in $Z^0(mod-\catC)$. If the cone of $\alpha$ exists it is unique up to isomorphism. Similarly for $n\in \Z$ and $x\in \catC$ we say that $z$ is the shift of $x$ and we write $z=x[n]$ if $h^z\cong h^x[n]$ in $mod-\catC$. We say that the dg category $\catC$ is \textit{pre-triangulated} if it has all shifts and all cones. If $\catC$ is any dg category one can form a new dg category $\widehat{\catC}$ by adding formally all shifts and cones. The category $\widehat{\catC}$ is called the pretriangulated hull of $\catC$ and it comes with a fully faithful dg functor $\catC\to \widehat{\catC}$ which is universal among functors from $\catC$ to pretriangulated dg categories. See for example \cite{keller3} section 4.5. for further details.

\subsection{Finite group actions on dg categories} \label{Finite group actions on dg categories}
We will be interested in symmetric powers of dg categories. To define this we first recall the more general notion of finite group actions on dg categories and equivariant objects. When talking about an action of a finite group $G$ on a dg category $\catC$ there are different levels of strictness that one can require. The following definition suffices for the purpose of this paper.

\begin{deff}\label{groupaction}
An action of $G$ on $\catC$ is the data of an autoequivalence $\rho_g:\catC\overset{\sim}{\to} \catC$ for every $g\in G$, natural isomorphisms $\theta_{g,g'}:\rho_{g}\circ \rho_{g'}\overset{\sim}{\implies}\rho_{g'g}$ and $\eta:\rho_e\overset{\sim}{\implies} \id_\catC$ and these should satisfy the following two conditions.\\

i) For each $g\in G$ and $c\in \catC$ we have
$$(\theta_{e,g})_c=\eta_{\rho_g(c)}:\rho_e\rho_g(c)\to \rho_g(c)$$
and  
$$(\theta_{g,e})_c=\rho_g(\eta_{c}):\rho_g\rho_e(c)\to \rho_g(c).$$

ii) For all $g,h,k\in G$ we have commutative diagrams
$$\begin{tikzcd}\rho_g\rho_h\rho_k(c)\arrow{r}{\rho_g((\theta_{h,k})_c)}\arrow{d}{(\theta_{g,h})_{\rho_k(c)}}&\rho_{g}\rho_{kh}(c)\arrow{d}{(\theta_{g,kh})_c}\\
\rho_{hg}\rho_k(c)\arrow{r}{(\theta_{hg,k})_c} &\rho_{khg}(c)	
\end{tikzcd}
$$
\end{deff}
\begin{ex}\label{symmetricgroupaction}
Let $\catC$ be a dg category. Then the symmetric group $S_n$ acts on the tensor power $\catC^{\otimes n}$ in the above sense. In fact, this action is strict in an even stronger sense; all the higher categorical data is trivial $\theta=\id$, $\eta=\id$.
\end{ex}

Whenever we have a group action as above we can consider the category of $G$-equivariant objects in $\catC$. 
\begin{deff}\label{equivariantobject}
	Let $(\rho,\theta,\eta)$ be an action of a finite group $G$ on a small dg category $\catC$. The dg category of equivariant objects, denoted $\catC^G$ has objects 
	$$(c,(\alpha_g:c\overset{\sim}{\to}\rho_g(c))_{g\in G})$$
	where the $\alpha_g$ are isomorphisms that are required to satisfy the following condition. For any $c\in \catC$ and $g,h\in G$ the following diagram
	$$\begin{tikzcd}
		c\arrow{r}{\alpha_g}\arrow{d}{\alpha_{hg}}&\rho_g(c)\arrow{d}{\rho_g(\alpha_h)}\\
		\rho_{hg(c)}\arrow{r}{(\theta_{g,h}^{-1})_c}&\rho_g\rho_h(c)
	\end{tikzcd}$$
	commutes.
	
	A degree $n$ morphism between equivariant objects $\varphi:(c,\alpha)\to (c',\alpha')$ is a degree $n$ morphism in $\catC$ such that 
	$$\begin{tikzcd}
		c\arrow{r}{\varphi}\arrow{d}{\alpha_g}&c'\arrow{d}{\alpha_g'}\\
		\rho_g(c)\arrow{r}{\rho_g(\varphi)}&\rho_g(c').
	\end{tikzcd}$$
\end{deff}

If $H\subset G$ is a subgroup of $G$ and $\catC$ has finite direct sums, then we can define restriction and induction functors 
$$\Res^G_H:\catC^G\to \catC^H, \ \Ind_H^G:\catC^H\to \catC^G.$$
The restriction functor just forgets part of the equivariant structure. The induction functor is slightly more involved and we will describe it explicitly. Let $\{r_1,...,r_n\}$ be complete set of representatives for $H\setminus G$ so that any $g\in G$ can be uniquely written as $h\cdot r_i$ with $h\in H$ and $1\leq i\leq n$. The right action of $G$ on $H\setminus G$ gives rise to a right action of $G$ on $\{1,2,...,n\}$. Then we define
$$\Ind_H^G(c,\alpha):=\Big(\bigoplus_{i=1}^n\rho_{r_i}(c),\tilde{\alpha}\Big)$$
where the equivariant structure is defined as follows: For $g\in G$ $r_ig=h_ir_{g(i)}$. Then $\tilde{\alpha}_g$ is defined as the composite
$$\begin{tikzcd}\bigoplus_{r=1}^n\rho_{r_i}(c)\arrow{rr}{\oplus \rho_{r_i}(\alpha_{g^{-1}(i)})}&&	\bigoplus_{i=1}^n \rho_{r_{i}}\rho_{h_{g^{-1}(i)}}(c)\cong \bigoplus_{i=1}^n\rho_g\rho_{r_{g^{-1}(i)}}(c)\arrow{r}{}&\bigoplus_{i=1}^n\rho_g\rho_{r_i}(c)
\end{tikzcd}
$$
where the last map is given by the permutation matrix $(\delta_{i=g^{-1}(j)})_{ij}$.

\subsection{Hochschild homology}\label{hh}
Given a dg category $\catC$ we have the left and right \textit{diagonal modules} over $\catC^e:=\catC\otimes\catC^{op}$ which we denote by $\Delta^l_\catC$ and $\Delta^r_\catC$. They are defined by
$\Delta^l_\catC(c_1,c_2):=\catC(c_2,c_1)$ and $\Delta^r_\catC(c_1,c_2):=\catC(c_1,c_2)$. For any $\catC\otimes \catC$ bimodule $M$ we define the \textit{Hochschild homology} of $\catC$ with values in $M$ as the derived tensor product $HH_\bullet(\catC,M):=M\overset{L}{\otimes_{\catC^e} }\Delta^l_\catC$. This derived tensor product can be computed by replacing $\Delta_\catC^l$ by a semi-free (or more generally an h-projective) left $\catC\otimes\catC^{op}$ module $P$ and then taking the usual tensor product $M\otimes_{\catC^e}P$. 

We now describe a general method to replace a $\catC$-module by a semi-free module. First, a \textit{dg-resolution} of $M$ is a sequence in $Z^0(\catC-mod)$ 
$$...\to F^{-2}\to F^{-1}\to F^0\to M\to 0$$
such that for all $c\in\catC$ the following three sequences are exact
\begin{align*}
    1)& ...\to F^{-2}(c)\to F^{-1}(c)\to F^0(c)\to M(c)\to 0\\
    2)& ...\to Z^\bullet (F^{-2}(c))\to Z^\bullet (F^{-1}(c))\to Z^\bullet (F^0(c))\to Z^\bullet (M(c))\to 0\\
    3)& ...\to H^\bullet(F^{-2}(c))\to H^\bullet(F^{-1}(c))\to H^\bullet(F^0(c))\to H^\bullet(M(c))\to 0.
\end{align*}
These three sequences being exact imply that the map $F^0\to M$ induces a quasi isomorphism $Tot^{\oplus }(F^\bullet)\overset{\sim}{\to} M$. Moreover, if all the $F^i$'s were free modules then $Tot^{\oplus }(F^\bullet)$ will be semi-free. Any module $M$ admits a free resolution as above (see for example section 2.1 in \cite{pol-pos}). 

We will only need the following standard resolution of the diagonal bimodule $\Delta_\catC^l$.
First we define 
$$\text{Bar}^n(\catC):=\Big((x,y)\mapsto\bigoplus_{c_0,...,c_{n}\in \catC}\catC(c_0,x)\otimes\catC(c_1,c_0)\otimes...\otimes \catC(y,c_{n})\Big)$$
for $n\geq 1$.
For each $n\geq 1$ we have the map of left $\catC\otimes \catC^{op}$ modules $d^n:\text{Bar}^n(\catC)\to \text{Bar}^{n-1}(\catC)$ defined by
$$f_0|f_1|\cdots|f_{n+1}\mapsto \sum_{i=0}^{n+1}(-1)^{i}f_0|\cdots |f_if_{i+1}|\cdots |f_{n+1}$$
and we have the map $d^0:\text{Bar}^0(\catC)\to \Delta_{\catC}^l$ defined by $f_0|f_1\mapsto f_0f_1$. We set $\text{Bar}(\catC):=Tot^{\oplus}(\text{Bar}^\bullet(\catC))$. The complex 
\begin{align}\label{bar}...\to \text{Bar}^2(\catC)\to \text{Bar}^1(\catC)\to \text{Bar}^0(\catC)\to \Delta_\catC^l\end{align}
is pointwise contractible via the contracting homotopy
$$f_0|\cdots|f_{n+1}\mapsto \id_x|f_0|\cdots |f_{n+1}.$$
Because of this it follows that the sequence (\ref{bar}) is in fact a dg-resolution of $\Delta_\catC^l$. More over each $\text{Bar}^n(\catC)$ is free as a $\catC\otimes \catC^{op}$-module.
Let us define $\text{Bar}(\catC):=Tot^\oplus (Bar^\bullet(\catC))$ which is a semi-free replacement of $\Delta_\catC^l$. 

If we use this particular semi-free resolution of $\Delta_\catC^{l}$ to compute the derived tensor product we obtain the \textit{standard complex} which is a particular representative of the object $HH_\bullet(\catC,M)$ in $D(k)$. Below we describe the standard complex in more detail.

\begin{deff}
	Let $\catC$ be a small dg category and let $M$ be a $\catC-\catC$-bimodule. Set 
	$$C_n(\catC,M):=\bigoplus_{c_0,...,c_n\in \catC}M(c_1,c_0)\otimes \catC(c_2,c_1)\otimes...\otimes \catC(c_0,c_n).$$
	Note that each summand above is a chain complex whose differential, denoted $d_1$, is obtained from the differentials in $\catC$ and the differential on the bimodule $M$. 
	We will use the bar notation for elements in the standard complex; to shorten notation we will denote a pure tensor $$a_0\otimes a_1\otimes\cdots\otimes a_n\in C_n(\catC,F)$$
by $a_0[a_1|\cdots |a_n]$.
	There is a second differential $d_{2}:C_{n}(\catC,F)\to C_{n-1}(\catC,F)$ defined by 
	\begin{align*}d_2\big(m[a_1|\cdots |a_n]\big)&=ma_1[a_2|\cdots|a_n]\\
	&+\sum_{i=1}^{n-1}(-1)^{i}m[\cdots|a_ia_{i+1}|\cdots]\\
	&+(-1)^{n+|a_n|(|a_0|+...+|a_{n-1}|)}a_nm[a_1|\cdots|a_{n-1}].
\end{align*}
The differentials $d_1$ and $d_2$ commute and therefore we get a total complex (cohomologically graded) by altering the sign of $d_1$ by $(-1)^n$ on $C_n(\catC,M)$ whose $k'$th term for $k\in \Z$ is
$$\bigoplus_{n-m=k}C_m(\catC,M)^n.$$
\end{deff}
\begin{rmk}
Different sign conventions appear for the Hochschild differential but they all give rise to isomorphic chain complexes and so we use this one as it is the least complicated. This sign convention appears for example in \cite{Lod}.
\end{rmk}

Hochschild homology with coefficients in an endofunctor (or bimodule more generally) have nice functorial properties as is explained in \cite{Pol}. We will use the notation from \cite{ville3} where this kind of functoriality is only described in terms of standard complexes.
Let $dg-cat^+$ denote the category whose objects are pairs $(\catC,F)$ where $\catC$ is a small dg category and $F:\catC\to \catC$ is an endo-functor and whose morphisms $(\catC,F)\to (\catC',F')$ consist of pairs $(\phi,\alpha)$ where $\phi:\catC\to \catC'$ is a dg functor and $\alpha:\phi F\implies F'\phi$ is a natural transformation. Composition is defined by $$(\phi_1,\alpha_1)\circ(\phi_2,\alpha_2):=(\phi_1\phi_2,\alpha_1\star\alpha_2)$$
where $\alpha_1\star\alpha_2:=(\alpha_1)_{\phi_2}\phi_1(\alpha_2)$. A pair $(\phi,\alpha)$ as above gives rise to a chain map between standard complexes $(\phi,\alpha)_*:C_\bullet(\catC,F)\to C_\bullet(\catC',F')$ and this makes $C_\bullet(-)$ and consequently $HH_\bullet(-)$ into functors $dg-cat^+\to D(k)$.

\begin{lemma}\label{trivialaction}
Let $F:\catC\to \catC$ be an endofunctor on a small dg category. Then $(F,\id)_*:HH_\bullet(\catC,F)\to HH_\bullet(\catC,F)$ is the identity. 
\end{lemma}
\begin{proof}
Consider the homotopy $H:C_\bullet(\catC,F)\to C_\bullet(\catC,F)$ defined by 
\begin{align*}
    a_0[a_1|\cdots |a_k]\mapsto &[a_0|\cdots |a_k]\\
    +&(-1)^{k+|a_k|(|a_0|+...+|a_{k-1}|)}[F(a_k)|a_0|\cdots|a_{k-1}]\\
    +&(-1)^{2k+(|a_{k-1}|+|a_k|)(|a_0|+...+|a_{k-2}|)}[F(a_{k-1})|F(a_k)|a_0|\cdots |a_{k-2}]\\
    +&...+(-1)^{k^2+(|a_1|+...+|a_{k}|)|a_0|}[F(a_1)|\cdots |F(a_k)|a_0].
\end{align*}
This satisfies $\bar{d}_1H+H\bar{d}_1=0$ and $d_2H+Hd_2=\mathbf{1}-(F,\id)_*$.

\end{proof}

Given two dg categories $\catC$ and $\catB$ with bimodules $M$ and $N$ we define 
	$$Sh:C_\bullet(\catC,M)\otimes C_\bullet(\catB,N)\to C_\bullet(\catC\otimes \catB,M\otimes N)$$
	by the formula
	\begin{align*}
m[f_1|\cdots |f_k]\otimes n[g_1|\cdots |g_l]\mapsto 	\quad\quad\quad\quad\quad\quad\quad\quad\quad\quad\quad\quad\quad\quad\quad\quad\quad\quad\quad\quad\quad\quad\\\sum_{\sigma\in (k,l)-shuffles}(-1)^{|\sigma|+\epsilon} m\otimes n[\cdots |f_{i}\otimes 1|\cdots|1\otimes g_j|\cdots|f_{i+1}\otimes 1|\cdots |1\otimes g_j|\cdots].
\end{align*}
The sign in the formula above is given by 
$$\epsilon=\sum |g_t||f_s|$$
where the sum is over all pairs $(f_s,g_t)$, $0\leq s\leq k$ and $0\leq t\leq l$, such that $f_s$ appears after $g_t$ (where $f_0=m$ and $g_0=n$).
An important tool when working with Hochschild homology is the Künneth isomorphism.
\begin{prop}\label{kunneth}
The map $Sh$ defined above induces an isomorphism
$$HH_\bullet(F,\catC)\otimes HH_\bullet(G,\catB)\to HH_\bullet(F\otimes G,\catC\otimes \catB).$$

When $\catC=\catB$ this isomorphism commutes with the $S_2$ action. \end{prop}
\begin{proof}
This is well known and a proof of a more general statement can be found in \cite{maclane} theorem VIII.8.1. We also included a proof of this, but for Hochschild homology without coefficients, in \cite{ville3} and the same argument goes through with coefficients.
\end{proof}

\subsection{Hochschild homology of equivariant dg category}\label{hhforeq}
In \cite{ville3} we proved that for a small dg category $\catC$ with finite direct sums we have
\begin{align}\label{equivariantdecomp}HH_\bullet(\catC^G)\cong \bigoplus_{[g]} HH_\bullet(\catC,\rho_g)^{C(g)}.\end{align}
(See also \cite{baranovsky}, \cite{Getzler-Jones} and \cite{Quddus}, for similar decompositions of the cyclic homology of a smooth variety, cross-product algebra and cross-product dg algebra respectively.)
On the right here, the action of $C(g)$ on $HH_\bullet(\catC,\rho_g)$ is defined by $C(g)\ni h\mapsto (\rho_h,S_{h,g})_*$ where $S_{h,g}$ is the natural isomorphism defined as the composite
$$\rho_h\rho_g\cong \rho_{gh}=\rho_{hg}\cong \rho_{g}\rho_h.$$
We note here that this action is the restriction of an action of $G$
 on $\bigoplus_{g\in G}HH_\bullet(\catC,\rho_g)$ the action of $h\in G$ restricted to the summand $HH_\bullet(\catC,\rho_g)$ is given by $(\rho_h,S_{h,g,h^{-1}gh})_*:HH_\bullet(\catC,\rho_g)\to HH_\bullet(\catC,\rho_{h^{-1}gh})$ and $S_{h,g,h^{-1}gh}$ is the natural transformation defined as a composite
 $$\rho_h\rho_g\cong \rho_{gh}=\rho_{h(h^{-1}gh)}\cong \rho_{h^{-1}gh}\rho_h.$$
 
 The projections associated with the direct sum decomposition \ref{equivariantdecomp},
$$\pi_g:HH_\bullet(\catC^G)\to HH_\bullet(\catC,\rho_g)^{C(g)}$$
are induced by the morphism $(U,\alpha_g)$ in $dg-cat^+$ where $\alpha_g(c):c\to \rho_g(c)$ is part of the $G$-equivariant structure of an object $(c,\alpha)\in \Ob(\catC^G)$. The inclusions
$$HH_\bullet(\catC,\rho_g)^{C(g)}\to HH_\bullet(\catC^G)$$
are induced by $(S,\phi_g)$ where 
$$\phi_g:\bigoplus_{h\in G}\rho_h\to \bigoplus_{h\in G}\rho_g\rho_{h}$$
is given by the matrix $(\delta_{h'g=h}\theta^{-1}_{h',g})_{hh'}$.

\section{Main theorem}\label{main theorem}
Let $\catC$ be any small dg category. Let $\catD:=\catC^{\otimes n}$.  Then there is a natural $S_n$ action on the category $\catD$. Let $\sigma:=(1 \quad 2\quad 3 \quad \cdots \quad n)\in S_n$. Let $\catD^e=\catD\otimes \catD^{op}$. We will use lowercase letters $x,y,...$ for objects in $\catD$ (or $\catC$) and primed letters $x',y',...$ for objects in $\catD^{op}$ (or $\catC^{op}$). Let $\Delta^r_{\sigma,\catC}$ (or simply $\Delta_\sigma^r$ when its clear which dg category is in play) be the right $\catD\otimes\catD^{op}$-module defined by $(x,y')\mapsto \catD(x,\sigma(y'))$. If $x=(x_1,x_2,...,x_n)$ and $y'=(y_1',...,y_n')$ with $x_i\in \catC$ and $y_i'\in \catC^{op}$ then 
$$\Delta_\sigma^r(x,y)=\catC(x_1,y_2')\otimes\catC(x_2,y_3')\otimes...\otimes \catC(x_n,y_1').$$
If $M$ is any left $\catD^e$-module and $N$ is a right $\catC^e$-module then we denote by $N\otimes_{\catC^e_{ij}}M$ the left $\catC^{\otimes n-1}\otimes (\catC^{op})^{n-1} $ module
$$((x_1,...,\hat{x_i},...,x_n),(y_1',...,\hat{y}_j',...,y_n))\mapsto N(\bullet,\bullet')\otimes_{\catC^e}M(x_1,...,\underset{\text{pos }i}{\bullet},...x_n,y_1',...,\underset{\text{pos }j}{\bullet'},...y_n').$$
We can repeat this: for $i'\neq i$ and $j'\neq j$ and two right $\catC^e$-modules $N_1,N_2$ we denote by $N_1\otimes_{\catC^e_{i'j'}}(N_2\otimes_{\catC^e_{ij}}M)$ the left $\catC^{\otimes n-2}\otimes (\catC^{op})^{n-2}$
$$((x_1,...,\hat{x}_{i'},...,\hat{x_i},...,x_n),(y_1',...,\hat{y}_j',...,\hat{y}_{j'},...,y_n))\mapsto $$
$$N_1(\star,\star')\otimes_{\catC^e}(N_2(\bullet,\bullet')\otimes_{\catC^e}M(x_1,...,\underset{\text{pos }i'}{\star},...,\underset{\text{pos }i}{\bullet},...x_n,y_1',...,\underset{\text{pos }j}{\bullet'},...,\underset{\text{pos }j'}{\star'},...y_n'))$$
and so on.
\begin{lemma}\label{keylemma}
   $(i)$ We have an isomorphism of functors
    $$\Delta_\sigma^r\otimes_{\catD^e}(-)\cong \Delta_\catC^r\otimes_{\catC^e_{12}}(\cdots(\Delta_\catC^r\otimes_{\catC^e_{n-1,n}}(\Delta_\catC^r\otimes_{\catC^e_{n1}}(-)))\cdots).$$

    $(ii)$ If $\phi:\catC\to \catC'$ is a dg functor then there is an isomorphism of functors
     $$\phi^*(\Delta_{\sigma,\catC'}^r)\otimes_{\catD^e}(-)\cong \phi^*(\Delta_{\catC'}^r)\otimes_{\catC^e_{12}}(\cdots(\phi^*(\Delta_{\catC'}^r)\otimes_{\catC^e_{n-1,n}}(\phi^*(\Delta_{\catC'}^r)\otimes_{\catC^e_{n1}}(-)))\cdots).$$
     Moreover, the diagram
     $$\begin{tikzcd}
         \Delta_\sigma^r\otimes_{\catD^e}(-)\arrow{r}{\cong}\arrow{d}{\phi\otimes \id_{(-)}}& \Delta_\catC^r\otimes_{\catC^e_{12}}(\cdots(\Delta_\catC^r\otimes_{\catC^e_{n-1,n}}(\Delta_\catC^r\otimes_{\catC^e_{n1}}(-)))\cdots)\arrow{d}{\phi\otimes(\phi\otimes(\cdots \otimes \id_{(-)}))}\\
         \phi^*(\Delta_{\sigma,\catC'}^r)\otimes_{\catD^e}(-)\arrow{r}{\cong}& \phi^*(\Delta_{\catC'}^r)\otimes_{\catC^e_{12}}(\cdots(\phi^*(\Delta_{\catC'}^r)\otimes_{\catC^e_{n-1,n}}(\phi^*(\Delta_{\catC'}^r)\otimes_{\catC^e_{n1}}(-)))\cdots)
     \end{tikzcd}$$
     commutes.
\end{lemma}
\begin{proof}
    Fix a left $\catD$-module $M$. Then both sides of the isomorphism is a quotient of 
    $$\bigoplus_{x_1,...,x_n\in\catC,y_1',...,y_n'\in \catC^{op}}\catC(x_1,y_2')\otimes\catC(x_2,y_3')\otimes...\otimes \catC(x_n,y_1')\otimes M(x_1,...,x_n,y_1',...,y_n').$$
      The maps in the following diagram are the projections associated with the various tensor products (here we have omitted the $\oplus$ to save space):
    $$\begin{tikzcd}
\catC(x_1,y_2')\otimes\catC(x_2,y_3')\otimes...\otimes \catC(x_n,y_1')\otimes M(x_1,...,x_n,y_1',...,y_n')\arrow{r}{\pi}\arrow{d}{p_n}&\Delta_\sigma^l\otimes_{\catD^e}M \\
\catC(x_1,y_2')\otimes\catC(x_2,y_3')\otimes...\otimes (\Delta_\catC^r\otimes_{\catC^e_{n1}} M(\bullet,...,x_n,y_1',...,\bullet'))\arrow{d}{p_{n-1}}&\\
\vdots\arrow{d}{p_1}&\\
\Delta_\catC^{r}\otimes_{\catC^e_{12}}(\Delta_\catC^r\otimes_{\catC^e_{23}}(...\otimes_{\catC^e_{n-1,n}} (\Delta_\catC^r\otimes_{\catC^e_{n1}} M)\cdots))
\end{tikzcd}$$
We will show that $\ker(\pi)=\ker(p_1p_2\cdots p_n)$. If $t\in \ker(\pi)$ then it can be written as a sum of elements of the form
$$\delta((\alpha_1,...,\alpha_n)\otimes (f_1,...,f_n,g_1',...,g_n)\otimes m):=$$
$$(g_2'\alpha_1f_1,g_3'\alpha_2f_2,...,g_1'\alpha_nf_n)\otimes m-(\alpha_1,\alpha_2,...,\alpha_n)\otimes (f_1\otimes...\otimes f_n)m(g_1'\otimes...\otimes g_n').$$
Now 
\begin{align*}p_1\cdots &p_n((g_2'\alpha_1f_1,g_3'\alpha_2f_2,...,g_1'\alpha_nf_n)\otimes m)\\
=&p_1\cdots p_{n-1}((g_2'\alpha_1f_1,g_3'\alpha_2f_2,...,\alpha_n)\otimes (1\otimes  ...\otimes 1\otimes f_n)m(g_1'\otimes 1\otimes...\otimes 1))\\
=&...\\
=&(\alpha_1,\alpha_2,...,\alpha_n)\otimes (f_1\otimes...\otimes f_n)m(g_1'\otimes...\otimes g_n')\end{align*}
This proves that $p_1\cdots p_n(t)=0$.

Conversely, suppose $p_1\cdots p_n(t)=0$. Then $\bar{t}:=p_2\cdots p_{n}(t)\in \ker(p_n)$. Therefore $\bar{t}$ can be written as a sum of elements of the form
$$[(f\alpha_1g,...,\alpha_n)\otimes m-(\alpha_1,...,\alpha_n)\otimes (f\otimes 1\otimes 1\otimes...\otimes 1)m(1\otimes g'\otimes...\otimes 1)]$$
where we use brackets to indicate that we are in the source of $p_1$ where elements already are equivalence classes. We can lift $\bar{t}$ to an element $t_1$ which is a sum of terms of the form
$$(f\alpha_1g',...,\alpha_n)\otimes m-(\alpha_1,...,\alpha_n)\otimes (f\otimes 1\otimes 1\otimes...\otimes 1)m(1\otimes g\otimes...\otimes 1)$$
(without the brackets now). Then $\pi(t_1)=0$ and $p_2\cdots p_{n}(t-t_1)=0$. Now we repeat this process to get elements $t_1,...,t_n\in \ker(\pi)$ and $t=t_1+...+t_n$. This completes the proof.

The proof of the isomorphism in part $(i)$ is identical to the first part. For the commutative diagram, note that we have a map
$$\begin{tikzcd}
    \underset{x_1,...,x_n\in\catC,y_1',...,y_n'\in \catC^{op}}{\bigoplus}\catC(x_1,y_2')\otimes...\otimes \catC(x_n,y_1')\otimes M(x_1,...,x_n,y_1',...,y_n')\arrow{d}{\phi\otimes\phi\otimes...\otimes \id_M}\\
    \underset{x_1,...,x_n\in\catC,y_1',...,y_n'\in \catC^{op}}{\bigoplus}\catC'(\phi (x_1),\phi(y_2'))\otimes...\otimes \catC'(\phi(x_n),\phi(y_1'))\otimes M(x_1,...,x_n,y_1',...,y_n').
\end{tikzcd}$$
Commutativity of the diagram now follows from the universal property of cokernels.
\end{proof}

Now we have the following key computation.

\begin{prop}\label{keycomputation}
    We have a natural isomorphism $HH_\bullet(\catC^{\otimes n},\sigma)\cong HH_\bullet(\catC)$.
\end{prop}
\begin{proof}
   We resolve $\Delta^l_\catD$ by the bar resolution $\text{Bar}(\catD)$ as in section \ref{hh}. To obtain $HH_\bullet(\catD,\sigma)$ we should apply $\Delta_{\sigma}^r\otimes_{\catD^e}(-)$ to $\text{Bar}(\catD)$. By the lemma this is isomorphic to what we get by first applying $\Delta_\catC^r\otimes_{\catC^e_{n1}}(-)$ and then $\Delta_\catC^r\otimes_{\catC^e_{n-1,n}}(-)$ and so on until finally applying $\Delta_\catC^r\otimes_{\catC^e_{12}}(-)$. \\
   
 \textit{Claim 1:} $\Delta_\catC^r\otimes_{\catC^e_{i,i+1}}(\cdots(\Delta_\catC^r\otimes_{\catC^e_{n-1,n}}(\Delta_\catC^r\otimes_{\catC^e_{n1}}(\text{Bar}^n(\catD)))\cdots)$ is free as $\catC^e_{i-1,i}$-module for $i=2,...,n$ (if $i=n$ then $i+1$ should be interpreted as $n$).\\

 \textit{Proof of claim 1}: Note that $\text{Bar}^n(\catD)$ is free as $\catD^e$-module. So it can be written as a direct sum of modules of the form $h_{(a_1,...,a_n,b_1',...,b_n')}\otimes V$ where $V$ is a chain complex. But by proposition \ref{proponoriginaltensorproduct} we then have
 $$\Delta_\catC^r\otimes_{\catC^e_{n1}}h_{(a_1,...,a_n,b_1',...,b_n')}\otimes V=\catC(a_n,b_1')\otimes h_{(a_1,...,a_{n-1},b_2',...,b_n')}\otimes V$$
 is still free as a $\catC^{\otimes n-1}\otimes(\catC^{\otimes n-1})^{op}$-module. If $n=2$ then we are done and if $n>2$ then we can just repeat the argument.\\

 \textit{Claim 2:} $\Delta_\catD^l$ is free as a $\catC^e_{n1}$-module. Moreover, $\Delta_\catC^r\otimes_{\catC^e_{i,i+1}}(\cdots(\Delta_\catC^r\otimes_{\catC^e_{n-1,n}}(\Delta_\catC^r\otimes_{\catC^e_{n1}}(\Delta_\catD^l))\cdots)$ is free as $\catC^e_{i-1,i}$-module for $i=3,...,n$ (for $i=n$ $i+1$ should be interpreted as $1$) and is isomorphic to $\Delta_\catC^l$ for $i=2$.\\ 
 
 \textit{Proof of claim 2:} Using proposition \ref{proponoriginaltensorproduct} we see that 
 \begin{align*}
     \Delta_\catD^l:(x_1,...,x_n,y_1',...,y_n')\mapsto &\catC(y_1',x_1)\otimes...\otimes \catC(y_n',x_n)\\
     \Delta_\catC^r\otimes_{\catC^e_{n1}}(\Delta_\catD^l)(x_1,...,x_{n-1},y_2',...,y_{n}')\mapsto &\\
     \catC(y_{n}',x_1)\otimes &\catC(y_2',x_2)\otimes...\otimes \catC(y_{n-1}',x_{n-1})\\
      \Delta_{\catC}^r\otimes_{\catC^e_{n-1,n}}(\Delta_\catC^r\otimes_{\catC^e_{n1}}(\Delta_\catD^l))(x_1,...,x_{n-2},y_2',...,y_{n-1}')\mapsto & \\
      \catC(y_{n-1}',x_1)\otimes &\catC(y_2',x_2)\otimes...\otimes \catC(y_{n-2}',x_{n-2})\\
      \vdots &\\
      \Delta_\catC^r\otimes_{\catC^e_{2,3}}(\cdots(\Delta_\catC^r\otimes_{\catC^e_{n-1,n}}(\Delta_\catC^r\otimes_{\catC^e_{n1}}(\Delta_\catD^l))\cdots)(x_1,y_2')\mapsto&\catC(y_2',x_1)
 \end{align*}
 where at each stage, the $\catC_{ij}^e$-module structure comes from simply thinking of the expressions on the right as a functor in $(x_i,y_j')$. From this the claim follows.\\

Note that tensoring commutes with forming direct sum totalizations. It then follows from the two claims, and lemma \ref{lemmaontensoringfree}, that $$\begin{tikzcd}Tot^{\oplus}(\Delta_\catC^r\otimes_{\catC^e_{2,3}}(\cdots(\Delta_\catC^r\otimes_{\catC^e_{n-1,n}}(\Delta_\catC^r\otimes_{\catC^e_{n1}}(\text{Bar}^\bullet(\catD)))\cdots))\arrow{d}{\sim}&\\ \Delta_\catC^r\otimes_{\catC^e_{i,i+1}}(\cdots(\Delta_\catC^r\otimes_{\catC^e_{n-1,n}}(\Delta_\catC^r\otimes_{\catC^e_{n1}}(\Delta_\catD^l))\cdots)\arrow{r}{\cong}&\Delta_\catC^l\end{tikzcd}$$
is a semi-free resolution of $\Delta_\catC^l$. Therefore after applying $\Delta_\catC^r\otimes_{\catC^e_{12}}(-)$ to 
$$Tot^{\oplus}(\Delta_\catC^r\otimes_{\catC^e_{2,3}}(\cdots(\Delta_\catC^r\otimes_{\catC^e_{n-1,n}}(\Delta_\catC^r\otimes_{\catC^e_{n1}}(\text{Bar}^\bullet(\catD)))\cdots))$$ we obtain $HH_\bullet(\catC)$.
\\

For naturality, suppose we have a dg functor $\phi:\catC\to \catC'$. We denote by $\catD=\catC^{\otimes n}$ and $\catD'=\catC'^{\otimes n}$. We have a commutative diagram (where the top two squares commute by lemma \ref{keylemma})
$$\begin{tikzcd}
\Delta^r_{\sigma,\catC}\otimes_{\catD^e}\text{Bar}^\bullet(\catD)\arrow{d}{}\arrow{r}{\cong}&\Delta_{\catC}^r\otimes_{\catC^e_{12}}(\Delta^r_{\catC}\otimes_{\catC_{23}^e}\cdots \otimes_{\catC^{e}_{n1}}\text{Bar}^\bullet(\catD))\arrow{d}{}\\
\Delta^r_{\sigma,\catC}\otimes_{\catD^e}\phi^*(\text{Bar}^\bullet(\catD'))\arrow{d}{}\arrow{r}{\cong}&\Delta_{\catC}^r\otimes_{\catC^e_{12}}(\Delta^r_{\catC}\otimes_{\catC_{23}^e}\cdots \otimes_{\catC^{e}_{n1}}\phi^*(\text{Bar}^\bullet(\catD')))\arrow{d}{}\\
\phi^*(\Delta^r_{\sigma,\catC'})\otimes_{\catD^e}\phi^*(\text{Bar}^\bullet(\catD'))\arrow{dd}{}\arrow{r}{\cong}&\phi^*(\Delta_{\catC'}^r)\otimes_{\catC^e_{12}}(\phi^*(\Delta^r_{\catC'})\otimes_{\catC_{23}^e}\cdots \otimes_{\catC^{e}_{n1}}\phi^*(\text{Bar}^\bullet(\catD')))\arrow{d}{}\\
&\phi^*(\Delta_{\catC'}^r)\otimes_{\catC^e_{12}}\phi^*(\Delta^r_{\catC'}\otimes_{(\catC')_{23}^e}\cdots \otimes_{(\catC')^{e}_{n1}}\text{Bar}^\bullet(\catD'))\arrow{d}{}\\
\Delta^r_{\sigma,\catC'}\otimes_{\catD'^e}\text{Bar}^\bullet(\catD')\arrow{r}{\cong}&\Delta_{\catC'}^r\otimes_{(\catC')^e_{12}}(\cdots \otimes_{(\catC')^{e}_{n1}}\text{Bar}^\bullet(\catD'))
\end{tikzcd}$$
where composing the maps on the left gives the induced map $HH_\bullet(\catC^{\otimes n},\sigma)\to HH_\bullet(\catC'^{\otimes n},\sigma)$ and composing the maps on the right gives the induced map $HH_\bullet(\catC)\to HH_\bullet(\catC')$.
\end{proof}
Now before we state and prove our main theorem we should explain the notation in the isomorphism (\ref{iso}) from the introduction. Given a dg category $\catC$ we denote by $\Sym^n(\catC)$ the dg category $\widehat{\catC^{\otimes n}}^{S_n}$ where the $S_n$ action on $\widehat{\catC^{\otimes n}}$ is an extension of the action described in example \ref{symmetricgroupaction}. The space $V:=\bigoplus_{i\geq 1}HH_\bullet(\catC)t^i$ is bigraded: elements in $HH_n(\catC)t^m$ are said to be in \textit{Hochschild degree} $n$ and $t-degree$ $m$. When we write $\Sym^\bullet(\bigoplus_{i\geq 1}HH_\bullet(\catC)t^i)$ we mean the graded symmetric powers where we only take the Hochschild degree into account. Since we are working over a field of characteristic zero we can either think of $\Sym^\bullet(V)$ as a quotient of the tensor algebra $T^\bullet(V)$ or as a subalgebra of $T^\bullet(V)$. To distinguish between these we use $S^\bullet(V)$ for the quotient and $\Sym^\bullet(V)$ for the subalgebra. 

For a moment, let $H:=HH_\bullet(\catC)$. For each $n\geq 0$ the following subspace of $T^\bullet(V)$ maps isomorphically onto the $t$-degree $n$ part of $S^\bullet(V)$:
\begin{align}\label{summands}\bigoplus_{a_1,...,a_n}\Sym^{a_1}(H)t^{a_1}\otimes \Sym^{a_2}(H)t^{2a_2}\otimes...\otimes \Sym^{a_n}(H)t^{na_n}\end{align}
where the direct sum runs over all $a_1,...,a_n\geq 0$ such that $\sum ia_i=n$.

\begin{theorem}\label{main theorem}
Let $\catC$ be a small dg category. Then we have a natural isomorphism of graded vector spaces
$$\bigoplus_{n\geq 0}HH_\bullet(\text{Sym}^n(\catC))t^n\cong S^\bullet(\oplus_{i\geq 1}HH_\bullet(\catC)t^i).$$
\end{theorem}

\begin{proof}
	First we make the identifications
	\begin{align}\label{equivariantdecomp}HH_\bullet(\Sym^n(\catC))t^n\cong \bigoplus_{\lambda\vdash n}HH_\bullet(\widehat{\catC^{n}},\sigma_\lambda)^{C(\sigma_\lambda)}t^n\cong \bigoplus_{\lambda\vdash n}HH_\bullet(\catC^{n},\sigma_\lambda)^{C(\sigma_\lambda)}t^n	
	\end{align}
where the first isomorphism is the main theorem in \cite{ville3} and the second is induced by the inclusions $\catC^n\to \widehat{\catC^n}$. We can describe the partition $\lambda$ either as $n_1+n_2+...+n_k=n$ with $1\leq n_1\leq n_2\leq...\leq n_k$ or as $n=1a_1+2a_2+...+na_n$ where $a_1,...,a_n$ are non-negative integers.
We can write $C(\sigma_\lambda)=C\rtimes S$ where $C=C(\lambda)=C_2^{\times a_2}\times C_3^{\times a_3}\times...\times C_n^{\times a_n}$ is a product of cyclic groups and $S=S(\lambda)=S_{a_1}\times S_{a_2}\times...\times S_{a_n}$ is a product of symmetric groups. For any $l$ let $\sigma_l=(1 \ 2 \ \cdots \ l)$ denote the cyclic permutation on $l$ ordered elements.
Then, by the Künneth isomorphism and the proposition above, we have
\begin{align*}HH_\bullet(\catC^n,\sigma_\lambda)t^n\cong &\big(HH_\bullet(\catC^{\otimes 1},\sigma_{1})t\big)^{\otimes a_1}\otimes...\otimes \big(HH_\bullet(\catC^{\otimes n},\sigma_{a_n})t^n\big)^{\otimes a_n}\\
	\cong &\big(HH_\bullet(\catC)t\big)^{\otimes a_1}\otimes...\otimes \big(HH_\bullet(\catC)t^n\big)^{\otimes a_n}.
\end{align*}
Note that by lemma \ref{trivialaction} $C\subset C\rtimes S$ acts trivially, so after taking invariants we get
\begin{align*}
	HH_\bullet(\catC^n,\sigma_\lambda)^{C(\sigma_\lambda)}\cong & \Big(\big(HH_\bullet(\catC)t\big)^{\otimes a_1}\otimes...\otimes \big(HH_\bullet(\catC)t^n\big)^{\otimes a_n}\Big)^{S}\\
	\cong & \Sym^{a_1}(HH_\bullet(\catC))t^{a_1}\otimes...\otimes \Sym^{a_n}(HH_\bullet(\catC))t^{na_n}.\\
\end{align*}
The last expression is a summand in (\ref{summands}). Moreover, as $n$ varies as well as $a_1,...,a_n$ we get all the all the summands (\ref{summands}). The direct sum of all these then map ismomorphically onto $S^\bullet(V)$ under the projection $T^\bullet(V)\to S^\bullet(V)$. Naturality follows from the fact that the identification (\ref{equivariantdecomp}), the Künneth isomorphism and the isomorphism in proposition \ref{keycomputation} is natural in $\catC$.

\end{proof}

\end{document}